\documentclass[12pt]{amsart}
\usepackage[margin=0in]{geometry}
\usepackage{latexsym}
\usepackage{amssymb}
\usepackage{amsmath}
\usepackage{amscd}
\usepackage{amsthm}
\usepackage{csquotes}
\usepackage{enumitem}
\usepackage[mathcal]{euscript}
\usepackage{hyperref}
\makeatletter

\allowdisplaybreaks
\theoremstyle{plain}
\newtheorem{defn}{Definition}
\newtheorem{thm}{Theorem}
\newtheorem{prop}{Proposition}
\newtheorem{lem}{Lemma}
\newtheorem{cor}{Corollary}
 
\newtheorem*{rem}{Remark}
\newtheorem*{obs}{Observation}
\newtheorem*{thm*}{Theorem}

\newcommand{\Hom}{\operatorname{Hom}}
\newcommand{\Biggg}{\bBigg@{3}}

\title{The Number of Monodromy Representations of Abelian Varieties of Low $p$-Rank}
\author{Brett Frankel}
\date{}

\begin{document}
\maketitle

\begin{abstract}
Let $A_g$ be an abelian variety of dimension $g$ and $p$-rank $\lambda \leq 1$ over an algebraically closed field of characteristic $p>0$. We compute the number of homomorphisms from $\pi_1^{\text{\'et}}(A_g,a)$ to $GL_n(\mathbb F_q)$, where $q$ is any power of $p$. We show that for fixed $g$, $\lambda$, and $n$, the number of such representations is polynomial in $q$, and give an explicit formula for this polynomial. We show that the set of such homomorphisms forms a constructible set, and use the geometry of this space to deduce information about the coefficients and degree of the polynomial.

In the last section we prove a divisibility theorem about the number of homomorphisms from certain semidirect products of profinite groups into finite groups. As a corollary, we deduce that when $\lambda=0$, 
\[\frac{\#\Hom(\pi_1^{\text{\'et}}(A_g,a),GL_n(\mathbb F_q))}{|GL_n(\mathbb F_q)|}\]
is a Laurent polynomial in $q$.
\end{abstract}

\tableofcontents

\newpage
\section*{Acknowledgments}
This paper is derived from my Ph.D. thesis, and as such there are far too many people who have impacted this work than can be listed here. It is a pleasure to thank Ted Chinburg, my thesis advisor, for suggesting this problem and for countless discussions and guidance. Thanks also to David Harbater, Zach Scherr, and Ching-Li Chai for helpful discussions and answering my questions. I am grateful to Professor Fernando Rodriguez-Villegas for hosting me for a very productive week at the International Centre for Theoretical Physics. It was he who observed Corollary~\ref{Laurent}, and  introduced me to both his theorem with Cameron Gordon \cite{GV} and the theorem of Frobenius, \cite{F} from which one easily deduces Corollary~\ref{Laurent}. He also made me aware of Proposition~\ref{int}. Bob Guralnick very helpfully referred me to both his paper with Sethuraman \cite{GS} and earlier work of Gerstenhaber \cite{G}. Additional thanks are due to Nir Avni for comments on the draft and Sebastian Sewerin for help reading \cite{F}. The exposition here is much improved thanks to the referee's careful reading. Many others have had a great influence on my mathematical development, which no doubt manifests itself in some of the work that follows. There are too many to list here, but I have attempted to acknowledge many of these people in the thesis document. \cite{thesis} This work was partially supported by NSF FRG grant 1265290 and NSA grant H98230-14-1-0145.

\section{Introduction} 

In \emph{Mixed Hodge Polynomials of Character Varieties}, \cite{HRV} Hausel and Rodriguez-Villegas computed the number of homomorphisms from the fundamental group of a Riemann surface $X$ of genus $g$ into $GL_n(\mathbb F_q)$. By van Kampen's theorem, this is equivalent to counting $2g$-tuples of matrices,
\begin{align*}P(q)&=\#\Hom(\pi_1(X,x),GL_n(\mathbb F_q))\\
&=\#\{(X_1,Y_1,\cdots,X_g,Y_g):[X_1,Y_1][X_2,Y_2]\cdots[X_n,Y_n]=1\},\end{align*}
\noindent where $[X_i,Y_i]$ denotes the commutator $XYX^{-1}Y^{-1}$. They also consider a twisted version of the problem, computing
\begin{align*}R(q)&=\#\Hom_{twist}(\pi_1(X,x),GL_n(\mathbb F_q))\\
&=\#\{(X_1,Y_1,\cdots,X_g,Y_g):[X_1,Y_1][X_2,Y_2]\cdots[X_n,Y_n]=\zeta_n\},\end{align*}
where $\zeta_n$ is a primitive $n^\text{th}$ root of $1$.

The salient feature of both these counts is that, for fixed $g$ and $n$, $P$ and $R$ are both polynomial functions of $q$. Another curious feature is that $P(q)$ is divisible by the order of $GL_n(\mathbb F_q)$. This phenomenon is explained in \cite{GV}.

In the twisted case, $PGL_n$ acts scheme-theoretically freely on $\Hom_{twist}(\pi_1(X),GL_n)$, and the GIT quotient $\mathcal{M}_n=\Hom_{twist}(\pi_1(X,x),GL_n)//PGL_n$ is the moduli space of twisted homomorphisms from $\pi_1(X,x)$ to $GL_n$. Via a theorem of N. Katz, (in the appendix to \cite{HRV}), the polynomial point-counting formula is equal to the $E$-polynomial $E(\mathcal M_n, T)$ of $\mathcal M_n(\mathbb C)$, which encodes information about the weight and Hodge filtrations on the cohomology of $\mathcal M_n$. Baraglia and Hekmati \cite{BH} have more recently extended these methods to the moduli of untwisted representations.

The results in this paper may be thought of as an analogue of the combinatorial results in \cite{HRV} and \cite{BH}, now applied to varieties of positive characteristic. The most natural analogue would be consider a curve $C/\bar{\mathbb F}_p$ in place of a Riemann surface. However, in positive characteristic, the \'etale fundamental group is a much more subtle object. In particular, there is not a single curve of genus greater than $1$ for which an explicit presentation of the fundamental group is known. \cite{PS}

However, these fundamental groups have well-understood abelianization; the abelianization of the fundamental group of a curve is the fundamental group of the Jacobian variety, which is the dual of the Jacobian. This paper treats representations of fundamental groups of abelian varieties. 

Let $A_g$ be an abelian variety of dimension $g$. The $p$-torsion points of $A$ form a vector space over $\mathbb F_p$ of dimension $0\leq \lambda \leq g$. We call $\lambda$ the $p$-rank of $A_g$. Let $a\in A$. The \'etale fundamental group of $A_g$ is given by

\begin{align*} \pi_1^{\text{\'et}}(A_g,a)&\cong \prod_{\ell\neq p} (\mathbb Z_\ell)^{2g} \times \mathbb Z_p^\lambda\\
&\cong (\hat{\mathbb Z}^\prime)^{2g-\lambda}\times \hat{\mathbb Z}^\lambda,
\end{align*}
\noindent where the above product is taken over all prime numbers $\ell$ other than $p$,  
\[\hat{\mathbb Z}=\varprojlim \mathbb Z/n\]
is the profinite completion of the group of integers, and 
\[\hat{\mathbb Z}^\prime=\varprojlim_{(n,p)=1} \mathbb Z/n\]
is the group of integers completed away from the prime $p$. To lighten the notation, we will henceforth write $\pi_1(A_g)$ instead of $\pi_1^{\text{\'et}}(A_g,a)$.

Let $q$ be a power of $p$. A homomorphism $\pi_1(A_g)\to GL_n(\mathbb F_q)$ is determined by the image of the topological generators, so specifying a homomorphism is equivalent to choosing an ordered $2g$-tuple of pairwise commuting matrices, such that the first $2g-\lambda$ have order not divisible by $p$.

Theorems~\ref{prank0}~and~\ref{prank1} explicitly compute the number of homomorphisms $\pi_1(A_g)\to GL_n(\mathbb F_q)$, where $q$ is a power of $p$ and the $p$-rank of $A_g$ is either $0$ or $1$, respectively. Theorem~\ref{conj} computes the number of homomorphisms up to conjugation in the $p$-rank $0$ case. All three counting formulas are polynomial in $q$, depending on $g$ and $n$ but not on the characteristic.

Section~\ref{space} considers the space of all such representation, and relates the geometry of this space to certain features of the polynomial formulas in the previous section. We show that the counting polynomial has integer coefficients, and give a lower bound for the degree, which we show to be exact when $A$ is in elliptic curve.

In section~\ref{divisibility}, we state and prove the following theorem.
\begin{thm*}[Theorem \ref{profiniteGV}]
Let $S$ be a set of primes (not necessarily finite), and let $\hat{\mathbb{Z}}_S=\varprojlim \mathbb{Z}/n$, where the the inverse limit is taken over all natural numbers $n$ not divisible by any prime in $S$. Then for any topologically finitely generated profinite group $\Gamma$ and finite group $G$,
 \[\frac{\#\Hom(\Gamma \rtimes \hat{\mathbb{Z}}_S, G)}{|G|}\in \mathcal{S}^{-1}\mathbb Z,\]
 where $\Gamma \rtimes \hat{\mathbb{Z}}_S$ is any semidirect product of $\Gamma$ and $\hat{\mathbb{Z}}_S$, and $\mathcal{S}$ is the multiplicative set generated by the elements of $S$. Conversely, if $\tilde{\Gamma}$ is topologically finitely generated and
 \[\frac{\#\Hom(\tilde{\Gamma}, G)}{|G|}\in \mathcal{S}^{-1}\mathbb Z\]
 for all finite groups $G$, then there exists a $\Gamma$ with $\tilde{\Gamma}\cong \Gamma \rtimes \hat{\mathbb{Z}}_S$.
\end{thm*}

Theorem~\ref{profiniteGV} is the profinite analogue of a divisibility theorem by Gordon and Villegas\cite{GV}. We deduce as a corollary that when the $\lambda=0$, $\Hom(\pi_1(A_g),GL_n(\mathbb F_q))/|GL_n(\mathbb F_q)|$ is a Laurent polynomial in $q$.

\section{Counting Formulas}\label{count}
Our main reference in this section is Macdonald's \emph{Symmetric Functions and Hall Polynomials}. \cite{Mac}
Throughout, $p$ will be a prime number, and $q$ will denote a power of $p$.
\subsection{Some Linear Algebra}

%Let $A$ be a supersingular abelian variety of dimension $g$ over an algebraically closed field of characteristic $p$. Then $\pi_1^{\text{\'et}}(A)\cong (\hat{\mathbb Z}^\prime)^{2g}$, where \[\hat{\mathbb Z}^\prime=\varprojlim_{(n,p)=1} \mathbb Z/n.\]
%Thus to count homomorphisms from $\pi_1^{\text{\'et}}(A)$ to $GL_n(\mathbb F_q)$, it suffices to count ordered $2g$-tuples of commuting matrices, with the requirement that each such matrix has order prime to $p$.

 For any matrix $X\in GL_n(\mathbb F_q)$ we have an associated $\mathbb F_q[T]$-module structure on $\mathbb F_q^n$, where $T$ acts by $X$.
 If $X$ and $Y$ induce isomorphic module structures on $\mathbb F_q^n$, then the isomorphism of $\mathbb F_q[T]$-modules defines an element of $GL_n(\mathbb F_q)$, so $X$ and $Y$ are conjugate.
Note that $X\in GL_n(\mathbb F_q)$ has order prime to $p$ if and only if $X$ acts semisimply on $\mathbb F_q^n$. That is, $X$ is diagonalizable over $\bar{\mathbb F}_q$. Since elements of order prime to $p$ act semisimply, the associated $\mathbb F_q[T]$-module is a direct sum of simple modules of the form $\mathbb F_q[T]/f(T)$, where the $f(T)$ are irreducible but not necessarily distinct.
 Matrices of order prime to $p$ are thus uniquely characterized, up to conjugation, by their characteristic polynomials.
 
 \subsection{Polynomials and their Types}

 \begin{defn} A \emph{type} $\Lambda$ of $n$ is a partition of $n$ along with a refinement of its conjugate. That is, a partition $\lambda$ of $n$ along with partitions $\lambda^i$ of the multiplicities of the entries $i$ of $\lambda$. \end{defn}
This is a slight generalization of what Macdonald calls a type in \cite{Mac}.
 For example, the data $\lambda=(5\ 3\ 3\ 3\ 3\ 3\ 2\ 2\ 2\ 1\ 1)=(5^{(1)}\ 3^{(5)}\ 2^{(3)}\ 1^{(2)})$, $\lambda^5=(1)$, $\lambda^4=\emptyset$, $\lambda^3=(2\ 2\ 1)$, $\lambda^2=(2\ 1)$, $\lambda^1=(2)$ give a type of $28$.
 We shall write $\lambda\vdash n$ when $\lambda$ is a partition of $n$, and $\Lambda \vDash n$ when $\Lambda$ is a type of $n$. %When taking products indexed over one of the partitions $\lambda^i$, we will do so with multiplicity. So if $\lambda^i=(2\ 2\ 1)$,
% \[\prod_{r\in\lambda^i} a_{i,r}=(a_{i,2})^2a_{i,1}.\]
 %On the other hand, a product $\prod_{i\in\Lambda} a_i$ shall denote a product over the entries of $\lambda$, taken without multiplicity.
 By \[\prod_{(i,r)\in\Lambda}f(i,r)\] we will mean the product taken over all pairs $(i,r)$ such that $i\in\lambda$ and $r\in\lambda^i$. We index over $\lambda$ without multiplicity, but over $r\in\lambda^i$ with multiplicity. So if
\[\Lambda=\begin{cases} \lambda=(3\ 3\ 3\ 3\ 3\ 1\ 1)\\
\lambda^3=(2\ 2\ 1)\\
\lambda^1=(2)\end{cases}\text{,}\]
then
\[\prod_{(i,r)\in\Lambda}f(i,r)=f(3,2)^2f(3,1)f(1,2).\]

 We associate to an $n\times n$ matrix $X$ a type $\Lambda\vDash n$ by considering the characteristic polynomial $c_X(T)$ of $X$. Factoring $c_X(T)$ into irreducible factors gives a partition of $\lambda\vdash n$, where the entries of $\lambda$ are the degrees of the irreducible factors of $c_X(T)$, counted with multiplicity. We then take $\lambda^i$ to be the partition consisting of the multiplicity of each distinct degree-$i$ factor of $c_X(T)$.
 
 For example, over $\mathbb F_3$, we associate $\lambda= (3\ 3\ 1\ 1\ 1\ 1\ 1\ 1)$, $\lambda^3=(2)$, $\lambda^1=(3\ 3)$ to the polynomial $(T^3+2T+1)^2(T-2)^3(T-1)^3$.
 
 Recall that the number of irreducible monic polynomials over $\mathbb F_q$ of degree $i$ is 
 \begin{equation*}\label{eq: 1}\frac{1}{i}\sum_{k\mid i} \mu(k)q^{i/k},\end{equation*}
 where $\mu$ is the M\"obius function. \cite{Lang}
\begin{defn} Denote by $\psi_{\Lambda}(q)$ the number of monic polynomials $p(T)\in \mathbb F_q(T)$ with factorization type $\Lambda$.\end{defn}
 Note that $\psi_{\Lambda}(T)\in\mathbb Q[T]$, but in general $\psi_{\Lambda}(T)\notin\mathbb Z[T]$. For instance, if $\lambda=(1\ 1)$, $\lambda^1=(1\ 1)$, then
\[\psi_{\Lambda}(q)=\frac{(q-1)(q-2)}{2}.\]

\subsection{Counting}\label{countingsection}

We now count the number of homomorphisms from $\pi_1(A_g)$ to $GL_n(\mathbb F_q)$, where the $p$-rank of $A_g$ is $0$. This is equivalent to counting $2g$-tuples of commuting matrices $(X_1,\ldots, X_{k})$ such that each $X_i$ has order prime to $p$. Since the combinatorial arguments below do not make use of the fact that $2g$ is even, we may state a slightly stronger theorem.
\begin{thm}\label{prank0} The number of ordered $k$-tuples of pairwise-commuting, semisimple, invertible matricies with entries in $\mathbb F_q$ is 
\begin{align*}|GL_n(q)|\sum_{\Lambda_1\vDash n}\psi_{\Lambda_1}(q)\prod_{(i_1,r_1)\in\Lambda_1}\sum_{\Lambda_2\vDash r_1}\psi_{\Lambda_2}(q^{i_1})
\prod_{(i_2,r_2)\in\Lambda_2}\cdots \\
 \cdots \sum_{\Lambda_{k-1}\vDash r_{k-2}}\psi_{\Lambda_{k-1}}(q^{i_1i_2\cdots i_{k-2}})\prod_{(i_{k-1},r_{k-1})\in\Lambda_{k-1}}\sum_{\Lambda_{k}\vDash r_{k-1}}\frac{\psi_{\Lambda_{k}}(q^{i_1i_2\cdots i_{k-1}})}
{\prod\limits_{(i_{k},r_k)\in\Lambda_{k}}|GL_{r_{k}}(q^{i_1i_2\cdots i_{k}})|}.
\end{align*}
\end{thm}

To illustrate, we consider the case $k=n=2$. Theorem~\ref{prank0} asserts number of commuting semisimple pairs $(X,Y)\in GL_2 (\mathbb F_q)^2$ is 
\[|GL_2(\mathbb F_q)|\frac{(q^3-q^2-q+1)}{q}=q^6-2q^5-q^4+4q^3-q^2-2q+1\text{,}\]
as follows.

There are three types of $2$, and one type of $1$.

\begin{tabular}{|c|c|c|c|}
\hline
$\Lambda$&$\lambda$ & $\{\lambda^i\}$& $\psi_{\Lambda}(q)$\\ \hline
$\Lambda_a\vDash 2$&$\lambda=(2)$ & $\lambda^2=(1)$&$\psi_{\Lambda_a}(q)=\frac{q^2-q}{2}$\\ \hline
$\Lambda_b\vDash 2$&$\lambda=(1\ 1)$& $\lambda^1= (2)$&$\psi_{\Lambda_b}(q)=(q-1)$\\ \hline
$\Lambda_c\vDash 2$&$\lambda=(1\ 1)$&$\lambda^1=(1\ 1)$&$\psi_{\Lambda_c}(q)=\frac{(q-1)(q-2)}{2}$\\ \hline
$\Lambda_d\vDash 1$&$\lambda=(1)$&$\lambda^1=(1)$& $\psi_{\Lambda_d}(q)=q-1$\\ \hline
\end{tabular}

The formula in Theorem~\ref{prank0} for $k=n=2$ is thus 
\small
\begin{align*} 
&|GL_2(\mathbb F_q)|\sum_{\Lambda \vDash 2} \psi_\Lambda(q)\prod_{(i,r)\in\Lambda}\sum_{N\vDash r}\frac{\psi_N(q^{i})}{\prod\limits_{(\ell,s)\in N}|GL_{s}(\mathbb F_{q^{i\ell}})|}\\
&\ \\
=&|GL_2(\mathbb F_q)|\left[\frac{(q^2-q)}{2}\prod_{(i,r)\in\Lambda_a}\sum_{N\vDash r}\frac{\psi_N(q^{i})}{\prod\limits_{(\ell,s)\in N}|GL_{s}(\mathbb F_{q^{i\ell}})|}\right.\\*
&+(q-1)\prod_{(i,r)\in\Lambda_b}\sum_{N\vDash r}\frac{\psi_N(q^{i})}{\prod\limits_{(\ell,s)\in N}|GL_{s}(\mathbb F_{q^{i\ell}})|}\\*
&+\left.\frac{(q-1)(q-2)}{2}\prod_{(i,r)\in\Lambda_c}\sum_{N\vDash r}\frac{\psi_N(q^{i})}{\prod\limits_{(\ell,s)\in N}|GL_{s}(\mathbb F_{q^{i\ell}})|}      \right]\\*
&\ \\
=&|GL_2(\mathbb F_q)|\left[\frac{(q^2-q)}{2}\sum_{N\vDash 1}\frac{\psi_N(q^{2})}{\prod\limits_{(\ell,s)\in N}|GL_{s}(\mathbb F_{q^{2\ell}})|}\phantom{{\left(\frac{q-1}{\prod\limits_{(\ell,s)\in \Lambda_d}|GL_{s}(\mathbb F_{q^{\ell}})|}\right)}^2}\right.\\*
&+(q-1)\sum_{N\vDash 2}\frac{\psi_N(q)}{\prod\limits_{(\ell,s)\in N}|GL_{s}(\mathbb F_{q^{\ell}})|}\\*
&+\left.\frac{(q-1)(q-2)}{2}{\left(\sum_{N\vDash 1}\frac{\psi_N(q)}{\prod\limits_{(\ell,s)\in N}|GL_{s}(\mathbb F_{q^{\ell}})|}\right)}^2 \right]\\*
&\ \\
=&|GL_2(\mathbb F_q)|\left[\left(\frac{q^2-q}{2}\right)\frac{\psi_{\Lambda_d}(q^{2})}{\prod\limits_{(\ell,s)\in \Lambda_d}|GL_{s}(\mathbb F_{q^{2\ell}})|}\phantom{{\left(\frac{q-1}{\prod\limits_{(\ell,s)\in \Lambda_d}|GL_{s}(\mathbb F_{q^{\ell}})|}\right)}^2}\right.\\*
&+(q-1)\left(\frac{\psi_{\Lambda_a}(q)}{\prod\limits_{(\ell,s)\in \Lambda_a}|GL_{s}(\mathbb F_{q^{\ell}})|}+\frac{\psi_{\Lambda_b}(q)}{\prod\limits_{(\ell,s)\in \Lambda_b}|GL_{s}(\mathbb F_{q^{\ell}})|}+\frac{\psi_{\Lambda_c}(q)}{\prod\limits_{(\ell,s)\in \Lambda_c}|GL_{s}(\mathbb F_{q^{\ell}})|}\right)\\*
&+\left.\frac{(q-1)(q-2)}{2}{\left(\frac{\psi_{\Lambda_d}(q)}{\prod\limits_{(\ell,s)\in \Lambda_d}|GL_{s}(\mathbb F_{q^{\ell}})|}\right)}^2      \right]\\
&\ \\
=&|GL_2(\mathbb F_q)|\left[\left(\frac{q^2-q}{2}\right)\frac{(q^{2}-1)}{\prod\limits_{(\ell,s)\in \Lambda_d}|GL_{s}(\mathbb F_{q^{2\ell}})|}\phantom{{\left(\frac{q-1}{\prod\limits_{(\ell,s)\in \Lambda_d}|GL_{s}(\mathbb F_{q^{\ell}})|}\right)}^2}\right.\\*
&+(q-1)\left(\frac{\left(\frac{q^2-q}{2}\right)}{\prod\limits_{(\ell,s)\in \Lambda_a}|GL_{s}(\mathbb F_{q^{\ell}})|}+\frac{q-1}{\prod\limits_{(\ell,s)\in \Lambda_b}|GL_{s}(\mathbb F_{q^{\ell}})|}+\frac{q-1}{\prod\limits_{(\ell,s)\in \Lambda_c}|GL_{s}(\mathbb F_{q^{\ell}})|}\right)\\*
&+\left.\frac{(q-1)(q-2)}{2}{\left(\frac{q-1}{\prod\limits_{(\ell,s)\in \Lambda_d}|GL_{s}(\mathbb F_{q^{\ell}})|}\right)}^2      \right]\\
&\ \\
=&|GL_2(\mathbb F_q)|\Biggg[\left(\frac{q^2-q}{2}\right)\frac{(q^{2}-1)}{q^2-1}+(q-1)\left(\frac{\left(\frac{q^2-q}{2}\right)}{q^2-1}+\frac{q-1}{(q^2-q)(q^2-1)}+\frac{\left(\frac{(q-1)(q-2)}{2}\right)}{{\left(q-1\right)}^2}\right)\\*
&+\frac{(q-1)(q-2)}{2}{\left(\frac{q-1}{q-1}\right)}^2   \Biggg]   \\
&\ \\
&=|GL_2(\mathbb F_q)|\frac{(q^3-q^2-q+1)}{q}\\*
&=q^6-2q^5-q^4+4q^3-q^2-2q+1\text{.}
\end{align*}
\normalsize
\noindent Similarly, the number of semisimple commuting pairs in $GL_3(\mathbb F_q)$ is 
\begin{align*}&| GL_3(\mathbb F_q)|\frac{q^6-q^5-q^4+2q^3-q^2+q-1}{q^3}\\
=&q^{12}-2 q^{11}-q^{10}+4 q^9-q^8-4 q^6+2 q^5+3 q^4-2 q^3+q^2-2 q+1\text{.}\end{align*}
\begin{proof}[Proof of Theorem \ref{prank0}]
 We first prove the theorem for $k=2$. Since semisimple matrices are characterized, up to conjugation, by their characteristic polynomials, we associate to each such conjugacy class a type $\Lambda$. There are by definition $\psi_\Lambda(q)$ conjugacy classes of type $\Lambda$.
 Suppose $X_1$ is any matrix of order prime to $p$, with associated type $\Lambda$. A matrix $Y$ commutes with $X_1$ if and only if $Y$ acts $X$-equivariantly on $\mathbb F_q^n$.
 That is, writing $\mathbb F_q^n$ as a sum of simple modules \[\mathbb F_q^n\cong\bigoplus_j(\mathbb F_q[T]/f_j(T))^{r_j}\text{,}\] the action of $Y$ is a $\mathbb F_q[T]$-automorphism of each $(\mathbb F_q[T]/f_j(T))^{r_j}$.
 Specifying such an action is given by an element of $GL_{r_j}(\mathbb F_{q^{\operatorname{deg} f_j}})$. The order of the centralizer of a semisimple matrix $X_1$ with type $\Lambda$ is thus
 \begin{equation}\label{eq: centX} \prod_{f_j}|GL_{r_j}(\mathbb F_{q^{\operatorname{deg} f_j}})|=\prod_{(i,r)\in\Lambda}|GL_{r}(\mathbb F_{q^i})|.\end{equation}

The matrix $X_2$ commutes with $X_1$ and also must be semisimple, so $X_2$ acts semisimply on each $(\mathbb F_q[T]/f_j(T))^{r_j}$. 
A semisimple matrix $Y_j \in GL_{r_j}(\mathbb F_{q^{\deg f_j}})$ is determined, up to conjugacy, by its characteristic polynomial, or equivalently by a type $N=(\nu,\nu^\ell)\vDash r_j$ and a polynomial of type $N$ in $\mathbb F_{q^{\deg f_j}}[T]$.
By the same argument used to compute the cardinality of the centralizer of $X_1$, we see that index of the centralizer of $Y_j\in GL_{r_j}(\mathbb F_{q^{\operatorname{deg} f_j}})$ is 
\[\frac{| GL_{r_j}(\mathbb F_{q^{\deg f_j}})|}{\prod\limits_{(\ell,s)\in N}|GL_{s}(\mathbb F_{q^{\deg (f_j)\ell}})|}\ \text{.}\]
The number of all such $Y_j$ is therefore 
\begin{equation}\label{eq: orbit}\sum_{N\vDash r_j}\left(\psi_N(q^{\deg f_j})\frac{| GL_{r_j}(\mathbb F_{q^{\deg f_j}})|}{\prod\limits_{(\ell,\nu)\in N}|GL_{s}(\mathbb F_{q^{\deg (f_j)\ell}})|}
\right).\end{equation}
So, if $X_1$ is semisimple with type $\Lambda$, the number of semisimple matrices $X_2$ commuting with $X_1$ is

\begin{align}
&\prod_{f_j} \sum_{N\vDash r_j}\left(\psi_N(q^{\deg f_j})\frac{| GL_{r_j}(\mathbb F_{q^{\deg f_j}})|}{\prod\limits_{(\ell,s)\in N}|GL_{s}(\mathbb F_{q^{\deg (f_j)\ell}})|}\right)\notag\\
=&\prod_{(i,r)\in\Lambda}\sum_{N\vDash r}\left(\psi_N(q^{i})\frac{| GL_{r}(\mathbb F_{q^i})|}{\prod\limits_{(\ell,s)\in N}|GL_{s}(\mathbb F_{q^{i\ell}})|}\right)\label{eq:ss}.
\end{align}

Given a semisimple matrix $X\in GL_n(\mathbb F_q)$, we have just computed both the number of matrices (\ref{eq: centX}) that commute with $X$ and the number of semisimple matrices (\ref{eq:ss}) that commute with $X$.
 Note that both of these numbers depend only on the type associated to $X$, since the $\operatorname{deg} (f_j)$'s are the entries of $\lambda$, and the $r_j$'s are the entries of the corresponding $\lambda^{\operatorname{deg}f_j}$'s.
 
 From these two computations, we see that the number of pairs $(X_1,X_2)$ of commuting semisimple matrices in $GL_n(\mathbb F_q)$ is
\begin{equation}\label{eq: 2}
 \sum_{\Lambda \vDash n} \sum_{\psi(\Lambda)} \frac{| GL_n(\mathbb F_q)|}{\prod\limits_{(i,r)\in\Lambda} |GL_{r}(\mathbb F_{q^i})|}\prod\limits_{(i,r)\in\Lambda}\sum_{N\vDash r}\left(\psi_N(q^{i})\frac{| GL_{r}(\mathbb F_{q^i})|}{\prod\limits_{(\ell,s)\in N}|GL_s(\mathbb F_{q^{i\ell}})|}\right)
 \end{equation}
 where the second sum is taken over polynomials with type $\Lambda$. Since 
 \[\frac{| GL_n(\mathbb F_q)|}{\prod\limits_{(i,r)\in\Lambda} |GL_{r}(\mathbb F_{q^i})|}\prod_{(i,r)\in\Lambda} \sum_{N\vDash r}\left(\psi_N(q^{i})\frac{ |GL_{r}(\mathbb F_{q^i})|}{\prod\limits_{(\ell,s)\in N}|GL_s(\mathbb F_{q^{i\ell}})|}\right)\]
depends only on $\Lambda$, we may simplify (\ref{eq: 2}) by replacing the second summation with multiplication by $\psi_\Lambda(q)$. Thus (\ref{eq: 2}) simplifies to 
 \begin{equation}\label{eq: 3}
  |GL_n(\mathbb F_q)|\sum_{\Lambda \vDash n} \psi_\Lambda(q)\prod_{(i,r)\in\Lambda}\sum_{N\vDash r}\frac{\psi_N(q^{i})}{\prod\limits_{(\ell,s)\in N}|GL_{s}(\mathbb F_{q^{i\ell}})|}.
\end{equation}
We now continue by induction. Suppose the number of pairwise commuting
\newline \noindent$(k-1)$-tuples of semisimple elements of $GL_n(\mathbb F_q)$ is
\begin{align*}|GL_n(q)|\sum_{\Lambda_1\vDash n}\psi_{\Lambda_1}(q)\prod_{(i_1,r_1)\in\Lambda_1}\sum_{\Lambda_2\vDash r_1}\psi_{\Lambda_2}(q^{i_1})\prod_{(i_2,r_2)\in\Lambda_2}\cdots \\
 \cdots \sum_{\Lambda_{k-2}\vDash r_{k-3}}\psi_{\Lambda_{k-2}}(q^{i_1i_2\cdots i_{k-3}})\prod_{(i_{k-2},r_{k-2})\in\Lambda_{k-2}}\sum_{\Lambda_{k-1}\vDash r_{k-2}}\frac{\psi_{\Lambda_{k-1}}(q^{i_1i_2\cdots i_{k-2}})}{\prod\limits_{(i_{k-1},r_{k-1})\in\Lambda_{k-1}} |GL_{r_{k-1}}(q^{i_1i_2\cdots i_{k-1}})|}\text{.}
\end{align*}
Suppose further that for each possible action of $X_1,\ldots,X_{k-1}$ on $\mathbb F_q^n$, the action of $X_{k-1}$ on an isotypic summands of the $\mathbb F_q[X_1,\ldots,X_{k-2}]$-module $\mathbb F_q^n$ has characteristic polynomials of type $\Lambda_{k-1}$ as above when these isotypic summands are viewed as $\mathbb F_{q^{i_1i_2\cdots i_{k-2}}}$-vector spaces. (We note that $\Lambda_{2k-1}$ may be different for different summands; indeed, $\Lambda_{k-1}\vDash r_{k-2}$, and the symbol $r_{k-2}$ takes on different values in different terms of the above formula).

Since $X_k$ commutes with each $X_\ell$, $1\leq \ell <k$, then $X_k$ acts on $\mathbb F_q^n$ by acting on each isotypic summand, which by induction are $\mathbb F_{q^{i_1i_2\cdots i_{k-2}}}[X_{k-1}]$-modules, or $\mathbb F_{q^{i_1i_2\cdots i_{k-1}}}$-vector spaces of dimension $r_{k-1}$. There are 
\[ \sum_{\Lambda_k\vDash r_{k-1}} \psi_{\Lambda_k}(q^{i_1i_2\cdots i_{k-1}}) \]

conjugacy classes of semisimple elements of $GL_{r_{k-1}}(\mathbb F_{q^{i_1i_2\cdots i_{k-1}}})$. By (\ref{eq: centX}), each conjugacy class has
\[ \frac{| GL_{r_{k-1}}(\mathbb F_{q^{i_1i_2\cdots i_{k-1}}})|}{\prod\limits_{(i_k,r_k)\in\Lambda_k}|GL_{r_k}(\mathbb F_{q^{i_1i_2\cdots i_k}})|}\] elements, for a total of 
\[ F(r_{k-1})=\sum_{\Lambda_k\vDash r_{k-1}} \psi_{\Lambda_k}(q^{i_1i_2\cdots i_{k-1}}) \frac{| GL_{r_{k-1}}(\mathbb F_{q^{i_1i_2\cdots i_{k-1}}})|}{\prod\limits_{(i_k,r_k)\in\Lambda_k}|GL_{r_k}(\mathbb F_{q^{i_1i_2\cdots i_k}})|}\]

\noindent possible $X_k$-actions on $\mathbb F_{q^{i_1i_2\cdots i_{k-1}}}^{r_{k-1}}$. So by induction, the number of pairwise-commuting $k$-tuples of semisimple matrices in $GL_n(\mathbb F_q)$ is
\begin{align*}|GL_n(q)|\sum_{\Lambda_1\vDash n}\psi_{\Lambda_1}(q)\prod_{(i_1,r_1)\in\Lambda_1}\sum_{\Lambda_2\vDash r_1}\psi_{\Lambda_2}(q^{i_1})\prod_{(i_2,r_2)\in\Lambda_2}\cdots \\
 \cdots 
\sum_{\Lambda_{k-1}\vDash r_{k-2}}\psi_{\Lambda_{k-1}}(q^{i_1\cdots i_{k-2}})\prod_{(i_{k-1},r_{k-1})\in\Lambda_{k-1}}\frac{1}{|GL_{r_{k-1}}(q^{i_1\cdots i_{k-1}})|}F(r_{k-1})
\text{.}
\end{align*}

Canceling factors of $|GL_{r_{k-1}}(q^{i_1\cdots i_{k-1}})|$, the above formula simplifies to the statement of the theorem. \end{proof}
\begin{obs} From the proof, we see that the formula in Theorem~\ref{prank0} is polynomial in $q$. This polynomial depends on $n$ and $k$, but not on the characteristic. \end{obs}

\begin{rem} The polynomial behavior of the formula in Theorem~\ref{prank0} is strictly a same-characteristic phenomenon, which fails even in the simplest case in characteristic $\ell\neq p$. Consider the set $\Hom(\hat{\mathbb Z}^\prime,\mathbb G_m(\mathbb F_{\ell^n}))$. When $\ell\not\equiv 1\mod p$, there are infinitely many $n$ for which $\ell^n\not\equiv 1\mod p$, and for such $n$ $\#\Hom(\hat{\mathbb Z}^\prime,\mathbb G_m(\mathbb F_{\ell^n}))=\ell^n-1$. But for those $n$ such that $\ell^n\equiv 1\mod p$, $\#\Hom(\hat{\mathbb Z}^\prime,\mathbb G_m(\mathbb F_{\ell^n}))\lneq\ell^n-1$. When $\ell\equiv 1\mod p$, we may write $\ell=mp^s+1$ with $p\nmid m$. Then for those $n$ such that $p\nmid n$, the $p$-adic valuation of $\ell^n-1$ is $s$, and $\#\Hom(\hat{\mathbb Z}^\prime,\mathbb G_m(\mathbb F_{\ell^n}))=(\ell^n-1)/p^s$. When $p\mid n$, $\#\Hom(\hat{\mathbb Z}^\prime,\mathbb G_m(\mathbb F_{\ell^n}))\lneq(\ell^n-1)/p^s$.
\end{rem}

The next theorem computes $\# \Hom(\pi_1(A_g),GL_n(\mathbb F_q))$ when the $p$-rank of $A_g$ is $1$.

\begin{thm}\label{prank1} The number of ordered $k$-tuples of invertible, pairwise-commuting matrices $(X_1,\ldots, X_{k-1},Y)$ such that the $X_i$ are all semisimple is 
\begin{align*}|GL_n(q)|\sum_{\Lambda_1\vDash n}\psi_{\Lambda_1}(q)\prod_{(i_1,r_1)\in\Lambda_1}\sum_{\Lambda_2\vDash r_1}\psi_{\Lambda_2}(q^{i_1})\prod_{(i_2,r_2)\in\Lambda_2}\cdots \\
 \cdots \sum_{\Lambda_{k-2}\vDash r_{k-3}}\psi_{\Lambda_{k-2}}(q^{i_1i_2\cdots i_{k-3}})\prod_{(i_{k-2},r_{k-2})\in\Lambda_{k-2}}
(q^{i_1\cdots i_{k-2}}-1)(q^{i_1\cdots i_{k-2}})^{r_{k-2}-1}.
\end{align*}
\end{thm}
\begin{proof} As in the inductive step of the theorem, we assume the number of pairwise commuting $(k-2)$-tuples of semisimple elements of $GL_n(\mathbb F_q)$ is
\begin{align*}|GL_n(q)|\sum_{\Lambda_1\vDash n}\psi_{\Lambda_1}(q)\prod_{(i_1,r_1)\in\Lambda_1}\sum_{\Lambda_2\vDash r_1}\psi_{\Lambda_2}(q^{i_1})\prod_{(i_2,r_2)\in\Lambda_2}\cdots \\
 \cdots \sum_{\Lambda_{k-3}\vDash r_{k-4}}\psi_{\Lambda_{k-3}}(q^{i_1i_2\cdots i_{k-4}})\prod_{(i_{k-3},r_{k-3})\in\Lambda_{k-3}}
\sum_{\Lambda_{k-2}\vDash r_{k-3}}\frac{\psi_{\Lambda_{k-2}}(q^{i_1i_2\cdots i_{k-3}})}{\prod\limits_{(i_{k-2},r_{k-2})\in\Lambda_{k-2}}|GL_{r_{k-2}}(q^{i_1i_2\cdots i_{k-2}})|}\text{.}
\end{align*}

As before, for each possible action of $X_1,\ldots,X_{k-2}$ on $\mathbb F_q^n$, the action of $X_{k-2}$ on an isotypic summands of the $\mathbb F_q[X_1,\ldots,X_{k-3}]$-module $\mathbb F_q^n$ has characteristic polynomials of type $\Lambda_{k-2}$ as above when these isotypic summands are viewed as $\mathbb F_{q^{i_1i_2\cdots i_{k-3}}}$-vector spaces.

Thus it suffices to count the number of pairs $X_{k-1},Y$ which act on each isotypic summand of the $\mathbb F_q[X_1,\ldots, X_{k-2}]$-module $\mathbb F^n_q$. By assumption, these summands are of the form $\mathbb F_{q^{i_1\cdots i_{k-2}}}^{r_{k-2}}$. 

 Let $\mathcal{O}$ be the orbit, under the conjugation action in $GL_{r_{k-2}}(\mathbb F_{q^{i_1\cdots i_{k-2}}})$, of a matrix of order prime-to-$p$. Given any $X_{k-1}$, we will denote by $A$ its restriction to the isotypic summand in question, and similarly $B$ will denote the restriction of $Y$. For any $A\in\mathcal O$, the number of nonsingular matrices $B$ that commute with $A$ is the order of the centralizer of $A$. The number of elements in $\mathcal O$ is the index of the centralizer of $A$. Thus there are precisely $|GL_{r_{k-2}}(\mathbb F_{q^{i_1\cdots i_{k-2}}})|$ pairs $(A,B)$ such that $A$ and $B$ commute and $A\in\mathcal O$.
 So to compute the number of pairs $(A,B)$ with $[A,B]=1$ and $|A|$ prime to $p$, we need only count the number of possible conjugacy classes $A$. These are in bijection with characteristic polynomials, of which there are $({q^{i_1\cdots i_{k-2}}})^{r_{k-2}-1}({q^{i_1\cdots i_{k-2}}}-1)$.

Thus the number of pairwise commuting $k$-tuples of invertible $n\times n$ matrices\newline \noindent $(X_1,\ldots, X_{k-1},Y)$, of which all but possibly the last have order prime to $p$, is
\begin{align*}|GL_n(q)|\sum_{\Lambda_1\vDash n}\psi_{\Lambda_1}(q)\prod_{(i_1,r_1)\in\Lambda_1}\sum_{\Lambda_2\vDash r_1}\psi_{\Lambda_2}(q^{i_1})\prod_{(i_2,r_2)\in\Lambda_2}\cdots \\
 \cdots \sum_{\Lambda_{k-3}\vDash r_{k-4}}\psi_{\Lambda_{k-3}}(q^{i_1\cdots i_{k-4}})\prod_{(i_{k-3},r_{k-3})\in\Lambda_{k-3}}
\sum_{\Lambda_{k-2}\vDash r_{k-3}}\frac{\psi_{\Lambda_{k-2}}(q^{i_1\cdots i_{k-3}})\#(A,B)}{\prod\limits_{(i_{k-2},r_{k-2})\in\Lambda_{k-2}}|GL_{r_{k-2}}(q^{i_1\cdots i_{k-2}})|}\text{,}
\end{align*}
where 
\[\#(A,B)=\prod_{(i_{k-2},r_{k-2})\in\Lambda_{k-2}}|GL_{r_{k-2}}(\mathbb F_{q^{}i_1\cdots i_{k-2}})|({q^{i_1\cdots i_{k-2}}})^{r_{k-2}-1}({q^{i_1\cdots i_{k-2}}}-1)\text{.}\]
 After cancellation we arrive at the statement of the theorem. \end{proof}

\begin{thm}\label{conj} The number of conjugacy classes of $k$-tuples of pairwise-commuting, invertible, semisimple matrices is
\begin{align*}\#\{(X_1,\ldots, X_k)\}/\sim=\\ \sum_{\Lambda_1\vDash n}\psi_{\Lambda_1}(q)\prod_{(i_1,r_1)\in\Lambda_1}\sum_{\Lambda_2\vDash r_1}\psi_{\Lambda_2}(q^{i_1})\prod_{(i_2,r_2)\in\Lambda_2}\cdots \\
 \cdots %\sum_{\Lambda_{k-2}\vDash r_{k-3}}\psi_{\Lambda_{k-2}}(q^{i_1\cdots i_{k-3}})\prod_{(i_{k-2},r_{k-2})\in\Lambda_{k-2}}
\sum_{\Lambda_{k-1}\vDash r_{k-2}}\psi_{\Lambda_{k-1}}(q^{i_1i_2\cdots i_{k-2}})\prod\limits_{(i_{k-1},r_{k-1})\in\Lambda_{k-1}}({q^{i_1\cdots i_{k-1}}})^{r_{k-1}-1}({q^{i_1\cdots i_{k-1}}}-1)\text{.}
\end{align*}
\end{thm}
\begin{proof}
Observe that $Y\in GL_n(\mathbb F_q)$ stabilizes $(X_1,\ldots, X_k)$ if and only if it commutes with each $X_\ell$. Writing $Z(X_1,\ldots,X_k)$ for the order of the stabilizer of $(X_1,\ldots, X_k)$ and applying the orbit-stabilizer lemma,
\begin{align*} \#\{(X_1,\ldots, X_k)\}/\sim &=\sum_{\{(X_1,\ldots, X_k)\}} \frac{Z(X_1,\ldots,X_k)}{|GL_n(\mathbb F_q)|}\\
&=\frac{1}{|GL_n(\mathbb F_q)|}\sum_{\{(X_1,\ldots, X_k)\}}Z(X_1,\ldots,X_k)\\
&=\frac{1}{|GL_n(\mathbb F_q)|}\sum_{\{(X_1,\ldots, X_k,Y)\}}1
\end{align*}
where the last sum is taken over $k+1$-tuples pairwise commuting semisimple matrices $(X_1,\ldots, X_k,Y)$ with all $X_\ell$ semisimple. The value of this sum was computed in Theorem~\ref{prank1}.
\end{proof}

\section{The Space of Representations}\label{space}
In this section, $A_g$ will be an abelian variety with $p$-rank $0$, as before. However, all statements apply if the $p$-rank of $A_g$ is $1$, mutatis mutandi. All schemes below are reduced. Recall that a subset of affine (or projective) space is said to be \emph{constructible} if it may be expressed as a Boolean combination of Zariski-closed sets. \cite{Marker}

Fix $n$, and let $R=\Hom(\pi_1(A_g), GL_n(\bar{\mathbb F}_q))$. Recall that a representation of $\pi_1(A_g)$ is given by a choice of $2g$ invertible matrices $X^k$, $1\leq i\leq 2g$, such that these matrices are pairwise commuting and each have order relatively prime to $p$. Assigning a matrix $(X^k_{ij})_{1\leq i,j\leq n}$ to each generator, we may view $R$ as a subset of $\mathbb A^{2gn^2}$.

We will now show that for fixed $n$ and $g$, the set $R=\Hom(\pi_1(A_g), GL_n(\bar{\mathbb F}_q))$ of all representations  may be considered a constructible set.

\begin{prop}\label{constr} $R$ is constructible.\end{prop}
\begin{proof} The requirement that the matrices pairwise commute defines a closed affine subscheme. Invertibility is an open condition, so the variety of pairwise commuting, invertible matrices is quasi-affine. We may further specify that the matrices are (simultaneously) diagonalizable with the statement that there exists an invertible matrix $C$ (with coordinates $C^{i,j}$, $1\leq i,j\leq n$) so that $CX^kC^{-1}$ is diagonal for each $k$. This last statement requires an existential quantifier, and thus does not necessarily define a subscheme, but it is a first order statement in the sense of model theory, and hence the space of representations is a definable subset of $\mathbb A^{2gn^2}$. Since the theory of algebraically closed fields admits quantifier elimination, \cite{Marker} every definable set is in fact constructible, i.e. a Boolean combination of closed subschemes.\end{proof}

Alternatively, since we are considering simultaneously diagonalizable matrices, one may write the diagonalization of each $X_k$ as an element of a torus, and then define a morphism $GL_n \times \mathbb G_m^{2gn}\to \mathbb A^{2gn^2}$, $(C,X_1,\ldots, X_{2g})\mapsto (C^{-1}X_1 C,\ldots,C^{-1}X_{2g}C)$. By Chevalley's Theorem, \cite{Marker} the image of this morphism, which is equal to $R$, is constructible.\qed

Since $R$ is constructible, we may write $R=(R_1-R^\prime_1)\cup (R_2-R^\prime_2)\cup\ldots \cup (R_t-R^\prime_t)$, where each $R_i$ is a projective variety and $R^\prime_i$ is a (possibly empty) closed subscheme of $R_i$. Alternatively, since $R$ is contained in affine space, we may choose each $R_i$ and $R_i^\prime$ to be affine.

\begin{prop}\label{int} Let $P(T)$ be the polynomial such that $P(q)=\#R(\mathbb F_q)$, the number of $\mathbb F_q$-rational points of $R$. Then $P(T)\in \mathbb Z[T]$.\end{prop}
This argument is not new, and can be found in the appendix of \cite{HRV} by Katz.
 \begin{proof} 
Chose $q$ so that each $R_i$ and $R_i^\prime$ is defined over $\mathbb F_q$. Let\newline
\noindent $P(T)=a_0+a_1T+\ldots a_dT^d$. Then the zeta function $Z(R,T)$ is defined to be 
\[Z(R,T)=\exp(\sum_{n=0}^{\infty} R(\mathbb F_{q^n})T^n/n)=\exp(\sum_{n=0}^{\infty} P(q^n)T^n/n)\]
On the other hand, each $R_i$ and $R_i^\prime$ is a projective variety, each $Z(R_i,T)$ and $Z(R_i^\prime,T)$ is rational, and
\[Z(R,T)=\prod_{i} \frac{Z(R_i,T)}{Z(R^\prime_i,T)}.\]
Writing
\[Z(R,T)=\frac{(1-\alpha_1 T)\cdots (1-\alpha_r T)}{(1-\beta_1 T)\cdots(1-\beta_s T)}\]
in lowest terms and taking logs, we see that
\begin{align*}\sum_{j=1}^r\log(1-\alpha_jT)-\sum_{k=1}^s\log(1-\beta_kT)&= \log Z(R,T)\\
&=\sum_{n=0}^{\infty} P(q^n)T^n/n\\
&=\sum_{n=0}^{\infty}\sum_{i=0}^d a_iq^{in}T^{n}/n.
\end{align*}
And so, taking derivatives,
\[\sum_{k=1}^s\frac{\beta_k}{1-\beta_kT}-\sum_{j=1}^r\frac{\alpha_j}{1-\alpha_jT}= \frac{\partial}{\partial T}Z(R,T)=\sum_{n=0}^{\infty}\sum_{i=1}^d a_iq^{in}T^{n-1}=\sum_{i=1}^d \frac{a_i q^i}{1-q^iT}.\]
Comparing poles, each $\alpha_j$ and $\beta_k$ must be a power of $q$. Comparing numerators we see that for each $i$, either $a_i$ is the number of $k$ such that $\beta_k=q^i$, or $-a_i$ is the number of $j$ for which $\alpha_j=q^i$.
\end{proof}

From the above proof, we see that the leading coefficient of $P(T)$ is the multiplicity of $(1-\beta_sT)$ in the denominator of $Z(R,T)$. 

\begin{thm}[R. Guralnick, B. Sethuraman \cite{GS}]\label{GS} Let $V$ be the variety of commuting $k$-tuples of $n\times n$ matrices. If $\Delta\subset V$ denotes the union of the discriminant loci, i.e. the set on which at least one matrix has a repeated eigenvalue, then the closure of $R-\Delta$ is an irreducible component of $V$. The dimension of this component is $n^2+(k-1)n$.
\end{thm}
\begin{cor}\label{degree} The degree of $P(T)$ is at least $n^2+(2g-1)n$.
\end{cor}
\begin{proof} Let $D\subset V$ be the union of the determinant loci, the set on which at least one matrix is not invertible. Note that $V\supset R\supset V-(\Delta\cup D)$. Then we may decompose $R=(R_1-R^\prime_1)\cup (R_2-R^\prime_2)\cup\ldots \cup (R_t-R^\prime_t)$, where $R_1$ is the closure of $V-\Delta$, $R_1^\prime=\Delta\cup D$, and all $R_i\subset (\Delta\cup D)$ for all $i>1$. By theorem~\ref{GS}, the closure of $R_1-\Delta$ is irreducible, and is thus equal to the closure of $R_1-(\Delta\cup D)$ because $D$ is closed. From the proof of Proposition~\ref{int}, the degree of $P(T)$ is $\sup_{i} \dim R_i$, and $\dim R_1=n^2+(k-1)$.
\end{proof}

\begin{thm}[Gerstenhaber \cite{G}]\label{pairs} The variety of commuting pairs of matrices is irreducible.
\end{thm}

\begin{cor} When $g=1$, $P(T)$ is monic of degree $n^2+(2g-1)n$.
\end{cor}
\begin{proof} Since the variety $V$ of pairs of commuting matrices is irreducible, we may write $R_1=V$ and $R_1^\prime=\Delta\cup D$, and $V\supset R\supset V-(\Delta\cup D)$ as in the proof of Corollary~\ref{degree}. All remaining $R_i$ are contained in $\Delta\cup D$, which is a closed subvariety of an irreducible variety and therefore of strictly lower dimension.
\end{proof}
\section{A Divisibility Theorem for Profinite Groups}\label{divisibility}

\subsection{Introduction and Statement of Theorem}
In ``On the divisibility of $\#\Hom(\Gamma, G)$ by $|G|$,'' \cite{GV} Cameron Gordon and Fernando Rodriguez-Villegas prove that for a finitely-generated group $\tilde{\Gamma}$, $\tilde{\Gamma}$ has infinite abelianization if and only if it satisfies the divisibility condition in the paper's title for all finite groups $G$.
This extends a result of Louis Solomon \cite{So} which proves the same, assuming $\tilde\Gamma$ has a presentation with more generators than relations. By \emph{homomorphism}, we will always mean a continuous homomorphism, where countable groups are given the discrete topology. We now state and prove the following profinite analogue of \cite{GV}:

\begin{thm}\label{profiniteGV} Let $S$ be a set of primes (not necessarily finite), and let $\hat{\mathbb{Z}}_S=\varprojlim \mathbb{Z}/n$, where the the inverse limit is taken over all natural numbers $n$ not divisible by any prime in $S$. Then for any topologically finitely generated profinite group $\Gamma$ and finite group $G$,
 \[\frac{\#\Hom(\Gamma \rtimes \hat{\mathbb{Z}}_S, G)}{|G|}\in \mathcal{S}^{-1}\mathbb Z,\]
 where $\Gamma \rtimes \hat{\mathbb{Z}}_S$ is any semidirect product of $\Gamma$ and $\hat{\mathbb{Z}}_S$, and $\mathcal{S}$ is the multiplicative set generated by the elements of $S$. Conversely, if $\tilde{\Gamma}$ is topologically finitely generated and
 \[\frac{\#\Hom(\tilde{\Gamma}, G)}{|G|}\in \mathcal{S}^{-1}\mathbb Z\]
 for all finite groups $G$, then there exists a $\Gamma$ with $\tilde{\Gamma}\cong \Gamma \rtimes \hat{\mathbb{Z}}_S$.
\end{thm}

\begin{rem} This result is precisely the profinite version of \cite{GV}, because a finitely generated group has infinite abelianiztion if and only if it is of the form $\Gamma\rtimes \mathbb Z$. \end{rem} 

\begin{cor}\label{Laurent} Let $P(T)$ be as in the previous section, i.e. the unique polynomial such that $P(q)=\Hom(\pi_1(A_g),GL_n(\mathbb F_q))$ whenever $q$ is a power of $p$, and denote by $G(T)$ the polynomial such that $G(q)=|GL_n(\mathbb F_q)|$. Then
 \[ \frac{P(T)}{|GL_n(\mathbb F_q)|}\in \mathbb Z[T,\frac{1}{T}].\]
\end{cor}
\begin{proof} Choose $S=\{p\}$. Writing $\pi_1(A_g)\cong (\hat{\mathbb Z}^\prime)^{2g}\cong (\hat{\mathbb Z}^\prime)^{2g-1}\times\hat{\mathbb Z}^\prime\cong(\hat{\mathbb Z}^\prime)^{2g-1}\times\hat{\mathbb Z}_S$, we see that for all $q$, $P(q)/|GL_n(\mathbb F_q)|$ is a rational number whose denominator is divisible only by the prime $p$. Thus $T^{n}\frac{P(T)}{G(T)}\in \mathbb Q[T]$ for sufficiently large $n$. Since $G(T)$ is monic, $T^{n}\frac{P(T)}{G(T)}\in \mathbb Z[T]$.
\end{proof}
\subsection{Proof (first statement)}
In this section we prove that for any topologically finitely generated $\Gamma$, the number of homomorphisms from $\Gamma \rtimes \hat{\mathbb Z}_S$ to $G$, when divided by $|G|$, has denominator divisible only by primes in $S$.

We first recall a theorem of Frobenius. \cite{F} (See \cite{IR} for a short elementary proof.)
\begin{thm}[Frobenius]\label{Frob}
  If $H$ is a finite group and $n\bigm\vert |H|$, then the number of elements of $H$ with order dividing $n$ is a multiple of $n$.
 \end{thm}
 
\begin{lem}\label{coset} Let $G$ be a finite group, and $H$ a subgroup of order $p^rm$, with $p\nmid m$. Choose $g\in N_G(H)$ with $|g|$ a power of $p$. Then the set $\{x\in Hg: |x| \text{ is a power of $p$}\}$ has cardinality divisible by $p^r$.
\end{lem}
\begin{proof} We may assume that $g\notin H$, for otherwise the lemma follows from Theorem~\ref{Frob}, and then reduce to the case that $G=\langle H, g\rangle$, since it suffices to prove the lemma for this subgroup. Let $p^s$ be the least positive integer with $g^{p^s}\in H$.

By Theorem \ref{Frob}, the number of elements of $G$ with $p$-power order is a multiple of $p^{r+s}$, and the number of such elements in $K=\langle H, g^p \rangle$ is a multiple of $p^{r+s-1}$. Thus the number of elements of $G\backslash K$ with $p$-power order is a multiple of $p^{r+s-1}$.

Let $x\in G\backslash K$ be of $p$-power order. Then we may write $x=yg^t$ with $y\in H$ and $p\nmid t$. The group $\mathbb Z/(p^{r+s})^\times$ acts on the set of all such $x$, $n\cdot x=x^n$.
The orbit of $x$ under this action is the set of generators of $\langle x\rangle$, and $\langle x \rangle$ surjects onto $\langle g\rangle/\langle g^{p^s}\rangle$. Each orbit inside $G\backslash K$ therefore contains the same number of elements in each coset of the form $Hg^u$ with $p\nmid u$. $G\backslash K$ is the union of $\varphi(p^s)=p^{s-1}(p-1)$ cosets of $H$, so the number of elements of $Hg$ with $p$-power order is equal to \newline
\nobreak{$\#\{x\in G\backslash K: |x|\text{ is a power of $p$}\}/\varphi(p^s)$}, and is thus divisible by $p^r$.
\end{proof}
For any $p\notin S$ we may write $\Gamma \rtimes \hat{\mathbb Z}_S\cong (\Gamma\rtimes\hat{\mathbb Z}_{S\cup\{p\}})\rtimes \mathbb Z_p$. Therefore, we have reduced the statement of the theorem to the following:

\begin{prop}
 Suppose $\Gamma$ is any topologically finitely generated profinite group, and let $v_p$ denote the $p$-adic valuation on the integers. Then
 \[ v_p\left(\frac{\#\Hom(\Gamma \rtimes \mathbb Z_p, G)}{|G|}\right)\geq 0.\]
\end{prop}
The proof of the proposition is now a straightforward modification of the argument in \cite{GV}.
\begin{proof}
 A homomorphism $\Phi :\Gamma \rtimes \mathbb{Z}_p$ is determined by its restriction $\phi=\Phi\Big{|}_\Gamma$ and the image of $1\in \mathbb Z_p$, which is an element $g\in G$ such that the order of $g$ is a power of $p$, subject to the condition that $\phi((-1)\gamma(1))=g^{-1}\phi(\gamma)g$ for all $\gamma\in\Gamma$. In particular, this condition ensures that $g$ normalizes $\phi(\Gamma)$, and thus normalizes the centralizer $C_\phi$ of $\phi(\Gamma)$.
 Observe that if a pair $(\phi, g)$ determines a well-defined homomorphism $\Gamma \rtimes \mathbb{Z}_p\to G$ as above, then so does $(\phi, x)$ if and only if $|x|$ is a power of $p$ and $x\in C_\phi g$.
 Thus \begin{align*} &\#\Hom(\Gamma \rtimes \mathbb{Z}_p,G)\\
=&\sum_\phi \#\{g \in G: |g|\text{ is a power of $p$},\  g^{-1}\phi(\gamma) g=\phi((-1)\gamma1)\ \forall\  \gamma\in \Gamma\},\end{align*}
 where the sum is taken over all $\phi$ which are restrictions of homomorphisms from $\Gamma\rtimes \mathbb Z_p$, ie. those $\phi$ for which there exists such a $g$.
 We let $G$ act by conjugation on the set of homomorphisms restricted to $\Gamma$. The stabilizer of $\phi$ under this action is $C_\phi$. 
 Denoting by the orbit of $\phi$ by $[\phi]$, each element of $[\phi]$ extends to the same number of homomorphisms on $\Gamma\rtimes \mathbb Z_p$. For each $[\phi]$ choose a representative $\phi$ and an element $g_\phi\in G$ such that $(\phi, g_\phi)$ determines a homomorphism.
  \begin{align*}
&\#\Hom(\Gamma \rtimes \mathbb{Z}_p,G)\\
=&\sum_{\phi} \#\{g \in G: |g|\text{ is a power of $p$},\ g^{-1}\phi(\gamma) g=\phi((-1)\gamma1)\ \forall\  \gamma\in \Gamma\}\\
=&\sum_{[\phi]} \frac{|G|}{|C_\phi|}\#\{g \in G: |g|\text{ is a power of $p$},\ g^{-1}\phi(\gamma) g=\phi((-1)\gamma1)\ \forall\  \gamma\in \Gamma\}\\
=&\sum_{[\phi]} \frac{|G|}{|C_\phi|}\#\{x \in C_\phi g_\phi: |x|\text{ is a power of $p$}\}.
 \end{align*}

The proposition follows by applying Lemma \ref{coset} to the coset $C_\phi g_\phi$.
 \end{proof}

\subsection{Proof of the Converse}
We have the following proposition:
\begin{lem}
 Let $S$ and $\mathcal{S}$ be as above, and $\tilde\Gamma$ a topologically finitely generated profinite group. Suppose that for all finite groups $G$, $\#\Hom(\tilde{\Gamma}, G)/|G|\in \mathcal{S}^{-1}\mathbb Z$. Then $\tilde\Gamma$ has
 \[\hat{\mathbb Z}_S=\varprojlim_{\substack{(n,p)=1\\ \forall p\in S}} \mathbb Z/n\] as a quotient.
\end{lem}
\begin{proof}
 Suppose $\tilde\Gamma$ does not surject onto \[\varprojlim_{\substack{(n,p)=1\\ \forall p\in S}} \mathbb Z/n.\]
 Then for some prime $\ell \notin S$ and sufficiently large $m$, $\tilde\Gamma$ does not surject onto $\mathbb Z/\ell^m$. 
 All homomorphisms from $\tilde\Gamma$ to abelian $\ell$-groups factor through the maximal pro-abelian $\ell$-quotient of $\tilde\Gamma$, which is itself a quotient of $(\mathbb Z/\ell^{m-1})^r$, where $\tilde\Gamma$ is topologically generated by $r$ generators. 
 But there are at most $\ell^{r(m-1)}$ homomorphisms $(\mathbb Z/\ell^{m-1})^r\to\mathbb Z/\ell^s$.
 Taking $s>r(m-1)$ and $G=\mathbb Z/\ell^s$, we see that $\#\Hom(\tilde\Gamma, G)$ is not a multiple of $|G|$, and thus is not an element of $\mathcal{S}^{-1}\mathbb Z$.
\end{proof}
All that remains to be shown is that $\tilde{\Gamma}$ is a semidirect product, which follows from the proposition below.
\begin{prop}
 Let $f: H\to \hat{\mathbb{Z}}_S$ be a continuous surjection of profinite groups. Then $f$ has a continuous section with closed image.
 \end{prop}
 \begin{proof}
Let $h\in H$ with $f(h)=1$. Consider the closed subgroup $\overline{<h>}$ generated by $h$. It is isomorphic to a product 
$\prod_{p\in I} \mathbb Z_p \times \prod_{p\notin I} \mathbb{Z}/p^{n_p}$, where $I$ is a set of primes.
 The order of $h$, as a supernatural number, is divisible by the order of $1\in \hat{\mathbb{Z}}_S$, and so $S$ is contained in the complement of $I$ (See \cite{Serre}). Thus $\overline{<h>}$ has a direct factor $K$ isomorphic to $\hat{\mathbb{Z}}_S$, which is a closed subgroup of $H$. $K$ is topologically generated by the image $k$ of $h$ under the projection map $\overline{<h>}\to K$, and $f(k)=1$. So $H$ has a closed subgroup $K$ which is abstractly isomorphic to $\hat{\mathbb{Z}}_S$, and $f$ sends a topological generator of $K$ to a topological generator of $\hat{\mathbb{Z}}_S$. The restriction of $f$ to $K$ is thus an isomorphism, and therefore admits a section.
\end{proof}

\begin{rem} We have in fact proven a stronger statement than the theorem. Namely, it suffices to check that $\#\Hom(\tilde\Gamma, G)/|G|\in\mathcal{S}^{-1}\mathbb Z$ only for groups of the form $G=\mathbb Z/{p^e}$, where $p\notin S$.\end{rem}

\bibliographystyle{plain}
\bibliography{AVMonodromyRepresentations}

\noindent \textsc{Department of Mathematics, Northwestern University.\\2033 Sheridan Road, Evanston, IL 60208, USA}\\
\noindent \emph{brett.frankel@northwestern.edu}

\copyright 2018. This manuscript version is made available under the CC-BY-NC-ND 4.0 license. \href{url}{http://creativecommons.org/licenses/by-nc-nd/4.0/}
\end{document}